\documentclass[a4paper,abstracton]{scrartcl}
\usepackage{fullpage}

\usepackage[T1]{fontenc}
\usepackage[utf8]{inputenc}
\usepackage[british]{babel}
\usepackage{csquotes}
\usepackage{amssymb}
\usepackage{amsmath}
\usepackage{amsthm}
\usepackage{stmaryrd}
\usepackage{hyperref}
\usepackage[plain]{fancyref}
\usepackage{xcolor}
\usepackage{microtype}
\usepackage{authblk}

\usepackage{newpxtext,newpxmath}
\setkomafont{disposition}{\bfseries}

\usepackage{algorithm}
\usepackage{algpseudocode}

\usepackage[style=alphabetic,natbib=true,maxbibnames=99,maxcitenames=2]{biblatex}
\hypersetup{colorlinks,
	linkcolor={blue!80!black},
	citecolor={blue!60!black},
	urlcolor={blue!80!black}
}

\DeclareMathOperator{\rank}{\mathrm{rank}}
\DeclareMathOperator{\fchar}{\mathrm{char}}
\DeclareMathOperator{\End}{\mathrm{End}}
\DeclareMathOperator{\Aut}{\mathrm{Aut}}
\DeclareMathOperator{\Fr}{Fr}
\DeclareMathOperator{\Tr}{Tr}

\DeclareMathOperator{\Gal}{Gal}
\DeclareMathOperator{\Jac}{Jac}
\DeclareMathOperator{\Hom}{Hom}
\DeclareMathOperator{\ord}{ord}
\DeclareMathOperator{\nrd}{nrd}
\DeclareMathOperator{\Spec}{Spec}
\newcommand{\Z}{\mathbb{Z}}
\newcommand{\Q}{\mathbb{Q}}
\newcommand{\R}{\mathbb{R}}
\newcommand{\C}{\mathbb{C}}

\newcommand{\F}{\mathbb{F}}
\renewcommand{\O}{\mathcal{O}}

\makeatletter
\def\mkfancyprefix#1#2{%
  \@namedef{fancyref#1labelprefix}{#1}%
  \begingroup\def\x{\endgroup\frefformat{plain}}%
    \expandafter\x\csname fancyref#1labelprefix\endcsname
      {\MakeLowercase{#2}\fancyrefdefaultspacing##1}%
  \begingroup\def\x{\endgroup\Frefformat{plain}}%
    \expandafter\x\csname fancyref#1labelprefix\endcsname
      {#2\fancyrefdefaultspacing##1}%
}
\makeatother

\newtheorem{theorem}{Theorem}
\mkfancyprefix{thm}{Theorem}
\newtheorem{definition}[theorem]{Definition}
\mkfancyprefix{def}{Definition}
\newtheorem{proposition}[theorem]{Proposition}
\mkfancyprefix{prop}{Proposition}
\newtheorem{lemma}[theorem]{Lemma}
\mkfancyprefix{lem}{Lemma}
\newtheorem{corollary}[theorem]{Corollary}
\mkfancyprefix{cor}{Corollary}
\newtheorem{remark}[theorem]{Remark}
\mkfancyprefix{rem}{Remark}
\newtheorem{example}[theorem]{Example}
\mkfancyprefix{ex}{Example}

\mkfancyprefix{table}{Table}
\mkfancyprefix{alg}{Algorithm}

\title{Abelian Varieties with $p$-rank Zero}
\author[1]{Yan Bo Ti}
\author[2]{Gabriel Verret}
\author[2]{Lukas Zobernig}
\affil[1]{DSO National Laboratories, Singapore \authorcr
\href{mailto:yanbo.ti@gmail.com}{\texttt{yanbo.ti@gmail.com}}}
\affil[2]{Department of Mathematics, University of Auckland, New Zealand \authorcr
\texttt{\{}\href{mailto:g.verret@auckland.ac.nz}{\texttt{g.verret}}, \href{mailto:lukas.zobernig@auckland.ac.nz}{\texttt{lukas.zobernig}}\texttt{\}@auckland.ac.nz}}
\date{\vspace{-5ex}}

\begin{document}
\maketitle

\begin{abstract}
There is a well known theorem by Deuring which gives a criterion for when the reduction of an elliptic curve with complex multiplication (CM) by the ring of integers of an imaginary quadratic field has ordinary or supersingular reduction. We generalise this and a similar theorem by Goren in dimension 2, and classify the $p$-torsion group scheme of the reduction of 3-dimensional abelian varieties with CM by the ring of integers of a cyclic sextic CM field. We also prove a theorem in arbitrary dimension $g$ that distinguishes ordinary and superspecial reduction for abelian varieties with CM by a cyclic CM field of degree $2g$.

As an application, we give algorithms to construct supersingular non-superspecial, and superspecial abelian varieties of dimension 2 (surfaces) and dimension 3, and show that all such varieties have non-integer endomorphisms of small degree.
\end{abstract}

\section{Introduction}
Fix some prime $p$ and consider elliptic curves defined over $\overline{\F}_{p}$. A general such curve $E$ will be \emph{ordinary} (i.e. the $p$-torsion of $E$ has $p$ many geometric points). The other possible case, called \emph{supersingular} (here the $p$-torsion just contains the point at the origin), is far more exceptional and in fact only occurs for finitely many elliptic curves (approximately $\lfloor p/12 \rfloor$ many), all of which are defined over $\F_{p^2}$. It is quite hard in practice to write down ``random'' supersingular elliptic curves, we either use known families of such curves (e.g. $y^2 = x^3 + 1$ if $p \equiv 2 \pmod{3}$ or $y^2 = x^3 + x$ if $p \equiv 3 \pmod{4}$) or the so called CM method to construct them from reductions of CM elliptic curves defined over some number field.

The correct way to generalise the term \emph{supersingular} to higher dimension abelian varieties is to consider those whose associated Newton polygon has slopes all $1/2$. We will see how this property is related to the $p$-torsion group scheme of abelian varieties in characteristic $p$ and their associated $p$-rank and $a$-number. For abelian surfaces, the $p$-rank directly informs about supersingularity, while in dimension 3 and higher it is no longer true that all abelian varieties with $p$-rank equal to zero are supersingular. Nevertheless, those \emph{strata} of the moduli space with $p$-rank zero are still interesting to study, and for example when the $a$-number equals the genus of the variety in question we find so called \emph{superspecial} abelian varieties that are isomorphic to a product of supersingular elliptic curves.

\paragraph{Motivation}
We are interested in constructing hyperelliptic curves whose Jacobian variety has $p$-rank zero. By dimension considerations of higher genus moduli spaces, we can likely only hope for families of such curves or to construct them using a generalisation of the CM method to higher dimension. Another necessary ingredient are reduction theorems that inform us about the resulting types of the corresponding $p$-torsion group schemes after reducing some chosen CM abelian variety.

\paragraph{Organisation}
In the remainder of this section we will recall some important notions: The structure of $p$-torsion group schemes, those parts of complex multiplication theory of abelian varieties necessary for us, and their associated reduction theory. An important part of this is the Shimura--Taniyama formula which we will use in the proofs of our theorems. In \Fref{sec:cyclic_cm_types} we study CM types of cyclic CM fields. \Fref{sec:higher_dim_reduction} is devoted to generalising the theorems of Deuring and Goren to dimension 3 and higher and the proofs of those theorems. In Sections \ref{sec:genus_2_construction} and \ref{sec:higher_genus_construction} we use the reduction theorems to give algorithms for constructing hyperelliptic $p$-rank zero curves and work through some explicit examples. Finally, in \Fref{sec:endomorphisms} we study the endomorphism rings of reductions of CM curves.

\subsection{The $p$-rank and $a$-number of Abelian Varieties and Curves}
Fix a prime $p$, and an algebraically closed field $k$ of characteristic $p$ containing $\F_p$. Consider the \emph{additive group scheme} $\mathbb{G}_a = \Spec k[X]$ and the \emph{multiplicative group scheme} $\mathbb{G}_m = \Spec k[X,X^{-1}]$. The kernel of the relative Frobenius on these yields the finite group schemes $\alpha_p \cong \Spec k[X]/X^p$ and $\mu_p \cong \Spec k[X]/(X^p-1)$, respectively. The Cartier dual of $\alpha_p$ is itself, and the Cartier dual of $\mu_p$ is the constant group scheme $\Z/p\Z$.

Let $A/k$ be an abelian variety of dimension $g$ and consider its $p$-torsion $A[p]$ as a group scheme. The \emph{$p$-rank} of $A$ is given by $f(A) = \dim_{\F_p} \Hom(\mu_p, A[p])$. Similarly, the $a$-number of $A$ is given by $a(A) = \dim_k \Hom(\alpha_p, A[p])$. Thus, geometrically, $A[p](k) \cong (\Z/p\Z)^f$. It holds that $0 \leq f \leq g$ and $1 \leq a+f \leq g$.

We define the $p$-rank and the $a$-number of a genus g curve $C$ as the corresponding invariants of its Jacobian $\Jac(C)$ as a $g$-dimensional abelian variety.

\begin{example}[\citet{pries2008}]\label{ex:genus_2_p_torsion}
Let $A/k$ be of dimension 2, i.e. an abelian surface, then we have the following possible types:
\begin{center}
\begin{tabular}{c|c|c|c|c}
$f(A)$ & $a(A)$ & $A[p]$ & Type & Codim. \\
\hline
$2$ & $0$ & $L^2$ & ordinary & 0 \\
$1$ & $1$ & $L \oplus I_{1,1}$ & non-ordinary & 1\\
$0$ & $1$ & $I_{2,1}$ & supersingular & 2 \\
$0$ & $2$ & $I_{1,1}^2$ & superspecial & 3
\end{tabular}
\end{center}
Here $L = \Z/p\Z \oplus \mu_p$ is the $p$-torsion of an ordinary elliptic curve, whereas $I_{1,1}$ is the $p$-torsion of a supersingular elliptic curve. It is also the unique local-local $BT_1$ group scheme of rank $p^2$ and fits into the following non-split exact sequence: $0 \to \alpha_p \to I_{1,1} \to \alpha_p \to 0$. Similarly, $I_{2,1}$ is the unique $BT_1$ group scheme of rank $p^4$ with $p$-rank 0 and $a$-number 1. The codimension of the associated strata in the full moduli space of abelian surfaces $\mathcal{A}_2$ is given as well. Since $\mathcal{A}_2$ is 3-dimensional, we can expect there to be a discrete number of superspecial points only.

Note that from how supersingularity is usually defined, a superspecial abelian surface is in particular supersingular. Hence, those abelian surfaces with $p$-torsion group scheme $I_{2,1}$ we should really call supersingular non-superspecial.
\end{example}

\begin{example}[\citet{pries2008}]\label{ex:genus_3_p_torsion}
Let $A/k$ be of dimension 3, then we have the following possible types:
\begin{center}
\begin{tabular}{c|c|c|c|c}
$f(A)$ & $a(A)$ & $A[p]$ & Type & Codim. \\
\hline
$3$ & $0$ & $L^3$ & ordinary & 0 \\
$2$ & $1$ & $L^2 \oplus I_{1,1}$ & non-ordinary (1) & 1 \\
$1$ & $1$ & $L \oplus I_{2,1}$ & non-ordinary (2) & 2 \\
$1$ & $2$ & $L \oplus I_{1,1}^2$ & non-ordinary (3) & 3 \\
$0$ & $1$ & $I_{3,1}$ & mixed (1) & 3 \\
$0$ & $2$ & $I_{3,2}$ & mixed (2) & 4 \\
$0$ & $2$ & $I_{1,1} \oplus I_{2,1}$ & supersingular & 5 \\
$0$ & $3$ & $I_{1,1}^3$ & superspecial & 6 \\
\end{tabular}
\end{center}
As mentioned in \citet{pries2019}, the codimension 3 and 4 strata with $A[p]$ torsion group scheme equal to $I_{3,1}$ and $I_{3,2}$, respectively, usually have \emph{Newton slopes} $(1/3,2/3)$, but can also have slopes all $1/2$ (only in this case the abelian variety $A$ is supersingular). On the other hand, the codimension 5 stratum only contains supersingular points.
\end{example}

\subsection{Complex Multiplication}\label{sec:complex_mult}
We call a number field $K$ a \emph{complex multiplication (CM) field} if $K$ is a totally imaginary quadratic extension of a totally real field $K_0$. The generator of $\Gal(K/K_0)$ is complex conjugation and commutes with all embeddings $\phi \in \Hom(K, \C)$; we denote its action on $x \in K$ by $\overline{x}$. Let $K$ be a CM field of degree $[K:\Q] = 2g$, then a \emph{CM type} of $K$ is a set $\Phi = (\phi_1,\dots,\phi_g)$ of embeddings $\phi_i: K \to \C$ such that $\Phi \cup \overline{\Phi} = \Hom(K,\C)$. Note that then no two embeddings $\phi_i$ and $\phi_j$ are conjugates of each other. We say that $\Phi$ is \emph{primitive} (sometimes also called \emph{simple}) if it is not the extension of a CM type on a CM subfield of $K$. Otherwise we say that the type is \emph{imprimitive}. If $K$ is normal, then a type $\Phi$ is primitive if and only if it is not fixed by any automorphism of $K$, i.e. if and only if there exists no $\gamma \in \Gal(K/\Q)$ such that $\Phi\gamma = \Phi$. We say that two CM types $\Phi$ and $\Phi'$ are \emph{equivalent} if there exists an automorphism $\gamma \in \Gal(K/\Q)$ such that $\Phi\gamma = \Phi'$.

Define the \emph{type trace} $T_\Phi(x) = \sum \phi_i(x)$ and \emph{type norm} $N_\Phi(x) = \prod \phi_i(x)$. Define the \emph{reflex field} $K'$ of $K$ (relative to a CM type $\Phi$) as the field generated over $\Q$ by all the elements $T_\Phi(x)$ with $x \in K$. Let $L$ be a finite Galois extension of $\Q$ containing $K$ with Galois group $G = \Gal(L/\Q)$, then we may lift the CM type $\Phi$ from $K$ to $L$, i.e. $\Phi_L = \{\sigma \in G \,|\, \sigma \text{ induces some } \phi_i \in \Phi\}$. Note that $\Phi^{-1}_L$ is a type lifted from the reflex field $K'$. Define the \emph{reflex type} $\Phi'$ on $K'$ corresponding to $\Phi$ on $K$ as the type which $\Phi^{-1}_L$ was lifted from. Similarly, we define the \emph{reflex norm} $N_{\Phi'}(x) = \prod \phi_i'(x)$. Note that $N_\Phi : K \to K'$, $N_{\Phi'} : K' \to K$, and that both are multiplicative maps that are well defined on ideals.

Let $\mathcal{O}$ be an order in $K$, for example the maximal order (ring of integers) $\mathcal{O}_K$ of $K$. Let $A$ be an abelian variety of dimension $g$, and denote by $\End(A)_\Q = \End(A) \otimes \Q$ the \emph{endomorphism algebra} of $A$. For our purposes, we say that $A$ has \emph{complex multiplication} by $(\mathcal{O},\Phi)$ if there exists an ideal $I \subset \mathcal{O}$ such that $A = \C^g/\Phi(I)$. See \citet[Section 1.3]{chaiconradoort2014} for a more general definition. Then there exists an embedding $\iota: K \to \End(A)_\Q$, and hence the data we should consider is the tuple $(A,\iota)$. We often assume the embedding $\iota$ to be given implicitly. We say that $A$ has CM by $(K,\Phi)$ if $A$ has CM by some order of $K$ and type $\Phi$. We call $A$ \emph{principal}, if its endomorphism ring $\End(A)$ is exactly $\mathcal{O}_K$ and not some smaller suborder.

\begin{theorem}[\cite{shta1961}, Section II.8.5 Proposition 30; \cite{lang1983}, Section 3.1 Theorem 1.1]\label{thm:reflex_field_of_definition}
Let $(A,\iota)$ be an abelian variety with CM by $(K,\Phi)$ and let $(K',\Phi')$ be the reflex field and type. If $(A,\iota)$ is defined over a number field $k$, then $K' \subset k$.
\end{theorem}

\begin{lemma}[\cite{shta1961}, Section II.8.4 Example (1)]\label{lem:abelian_ext_reflex_field} If $(K,\Phi)$ is a primitive CM field that is an abelian extension of $\Q$, then $K$ is its own reflex field. Furthermore, the reflex type $\Phi'$ of a type $\Phi$ is then given by $\Phi' = \Phi^{-1}$.
\end{lemma}

Let now $A$ be an abelian variety with CM by the ring of integers $\mathcal{O}_K$ of a CM field $K$. For an ideal $\mathfrak{a}$ of $\mathcal{O}_K$ let $A/\mathfrak{a}$ be the quotient of $A$ by the finite flat group scheme
\begin{equation*}
A[\mathfrak{a}] = \bigcap_{a \in \mathfrak{a}} A[a],
\end{equation*}
where $A[a] = \ker a$ for $a \in \End(A)$. One can show that there exists an isogeny $\lambda_\mathfrak{a}: A \to B$, where $B = A/\mathfrak{a}$, and that the kernel of $\lambda_\mathfrak{a}$ is given by $\ker \lambda_\mathfrak{a} = A[\mathfrak{a}]$, see \citet[Section II.7]{shta1961}. We call such an isogeny $\lambda_\mathfrak{a}$ an \emph{$\mathfrak{a}$-multiplication}. If $\mathfrak{a} = \mathfrak{a}_1 \mathfrak{a}_2$ for coprime ideals $\mathfrak{a}_1$ and $\mathfrak{a}_2$, then $A[\mathfrak{a}] = A[\mathfrak{a}_1] \oplus A[\mathfrak{a}_2]$.

\subsection{Explicit Construction of CM Abelian Varieties}\label{sec:cm_explicit_construction}
\citet[Chapter II.6]{shta1961} describe an explicit construction of abelian varieties with CM by an order $\mathcal{O}$ of a chosen CM field $K$. The following exposition is essentially proven as \citet[Theorem 3]{vanwamelen1999}, and considers principal abelian varieties.

Over the complex numbers $\C$, every $g$-dimensional abelian variety can be written as a quotient of the form $\C^g/\Lambda$, for some lattice $\Lambda \subset \C$. Fixing a CM field $K$ of degree 2g, let $\Phi$ be a CM type on $K$, and let $\mathcal{O}_K$ be the ring of integers of $K$. Let $\mathfrak{D}_{K/\Q}$ be the different of $K$, and choose some fractional $\mathcal{O}_K$-ideal $\mathfrak{a}$. If there exists some $\xi \in K$ such that the following holds:
\begin{enumerate}
\item $\xi$ is totally imaginary with $\Im(\phi(\xi)) > 0$ for all embeddings $\phi \in \Phi$,
\item $\xi\mathcal{O}_K = (\mathfrak{a}\overline{\mathfrak{a}}\mathfrak{D}_{K/\Q})^{-1}$,
\end{enumerate}
then for all $x,y \in K$ the map $(\Phi(x),\Phi(y)) \mapsto \Tr_{K/\Q}(\xi\overline{x}y)$ extends to a unique \emph{Riemann form} $E:\C^g \times \C^g \to \R$. The Riemann form $E$ then determines a \emph{principal polarisation} on the complex abelian variety $A = \C^g/\Phi(\mathfrak{a})$ such that $A$ has CM by $\mathcal{O}_K$. Two such principally polarised abelian varieties $(\C^g/\Phi(\mathfrak{a}),\xi)$ and $(\C^g/\Phi(\mathfrak{b}),\zeta)$ are isomorphic if there exists some $\gamma \in K^*$ such that $\gamma\mathfrak{a} = \mathfrak{b}$ and $\xi = \gamma\overline{\gamma}\zeta$.

\subsection{Reduction of CM Abelian Varieties}\label{sec:reduction_cm_abvar}
Consider an abelian variety $A$ over some number field $k$. If $\mathfrak{P}$ is a prime of $k$, denote the reduction of $A \mod \mathfrak{P}$ by $\tilde{A}$. We will use the ``reduction $A \mod \mathfrak{P}$'' as shorthand for $A \times \Spec \mathcal{O}_K/\mathfrak{P}$, the base change of $A$ to the residue field of $\mathfrak{P}$.
General reduction theory of abelian schemes shows that for primes $\mathfrak{P}$ of good reduction, we have a degree preserving injection of endomorphism rings
\begin{equation*}
\End(A) \hookrightarrow \End(\tilde{A}).
\end{equation*}
In general this is not a bijection, as one can immediately see from the example of supersingular reduction of CM elliptic curves: Such a curve $E$ over a number field has endomorphism ring an order $\mathcal{O}$ in an imaginary quadratic field, whereas for supersingular reduction the resulting elliptic curve $\tilde{E}$ over $\overline{\F}_p$ has endomorphism ring a maximal order of the definite quaternion algebra ramified at $p$. Hence $\End(\tilde{E})$ is strictly larger than $\End(E)$.

If $\mathfrak{P}$ is a prime of good reduction for $A$, then for an isogeny $\lambda:A \to B$ we find the reduced isogeny $\tilde{\lambda}: \tilde{A} \to \tilde{B}$. We can now ask what happens to the reduction of an $\mathfrak{a}$-multiplication. For an important special case that we will use, consider the following data:

\begin{itemize}
\item $(A,\iota)$ is an abelian variety defined over a number field $k$ with CM by $(K,\Phi)$, and is principal.
\item $(K',\Phi')$ is the reflex field and type of $(K,\Phi)$. By \Fref{thm:reflex_field_of_definition} we have $K' \subset k$.
\item $\mathfrak{P}$ is a prime of $k$ at which $A$ has good reduction, $\mathfrak{p} = \mathfrak{P} \cap \mathcal{O}_K$ is the prime of $K'$ dividing $\mathfrak{P}$, and $p = \mathfrak{p} \cap \Z$ is the prime of $\Q$ dividing $\mathfrak{p}$. Denote the reduction of $A \mod \mathfrak{P}$ by $(\tilde{A},\tilde{\iota})$.
\item As in \Fref{sec:complex_mult}, $L$ is a Galois extension of $\Q$ containing $K$. Let $\mathfrak{P}_L$ be an extension of $\mathfrak{P}$ to $L$, $\sigma = \sigma_{\mathfrak{P}_L}$ a Frobenius automorphism in $\Gal(L/K')$ corresponding to $\mathfrak{P}_L$.
\end{itemize}

\begin{theorem}\label{thm:shimura_taniyama_formula}
Let $(A,\iota)$ be an abelian variety along with the data as above. Assume that $\End(A) \cong \mathcal{O}_K$, and that $p$ is unramified in $K$.
\begin{enumerate}
\item The ideal $\mathfrak{a} = N_{\Phi'}(\mathfrak{p})$ induces the $\mathfrak{a}$-multiplication $\lambda : (A,\iota) \to (A^\sigma, \iota^\sigma)$ such that $\tilde{\lambda} = \Fr_\mathfrak{p}: \tilde{A} \to \tilde{A}^{(N(\mathfrak{p}))}$ is the $N(\mathfrak{p})$-power Frobenius homomorphism.
\item There exists and element $\pi_\mathfrak{P} \in \mathcal{O}_K$ such that $\tilde{\iota}(\pi_\mathfrak{P}) = \Fr_\mathfrak{P}$ is the $N(\mathfrak{P})$-power Frobenius endomorphism in $\End(\tilde{A})$. Then the ideal $(\pi_\mathfrak{P})$ has the factorisation $(\pi_\mathfrak{P}) = N_{\Phi'}(N_{k/K'}(\mathfrak{P}))$.
\end{enumerate}
\end{theorem}
\begin{proof}
This is \citet[Section III.13 Theorem 1]{shta1961}. It can also be found in \citet[Section 3.3]{lang1983}: 1. is Lemma 3.1 and Theorem 3.4, 2. is Lemma 3.3. Note that there is a typo in the statement of Lemma 3.3, the image of $\pi_\mathfrak{P}$ is rather the $N(\mathfrak{P})$-power Frobenius endomorphism $\Fr_{\mathfrak{P}}$ of $\tilde{A}$.
\end{proof}

\begin{remark}
Recall that for an abelian scheme $A$ over $\F_q$ (where $q = p^r$ is a prime power for some prime $p$), we denote by $A^{(p^s)} = A \times_{\Fr_{p^s}} \Spec \F_q$ the \emph{$p^s$-Frobenius twist} of $A$ as the base change by the $p^s$-power Frobenius morphism. We have $A^{(q)} = A$, but $A^{(p^s)}$ for some $s \neq r$ is only isogenous to $A$. Hence, $\Fr_q \in \End(A)$ is an endomorphism, but there are also the homomorphisms $\Fr_{p^s} : A \to A^{(p^s)}$ as isogenies in $\Hom(A,A^{(p^s)})$.
\end{remark}

By \citet{yu2004}, the Dieudonné module (and hence the $p$-rank and $a$-number) of the reduced abelian variety $\tilde{A}$ is uniquely determined by the CM type $\Phi$ and the splitting type of the prime $p$ in $K$. Say $p$ is unramified and decomposes as $p = \prod_{i=1}^n \mathfrak{p}_i$ with inertia degrees $f_i$.  We have $A[p] = \bigoplus_{i=1}^n A[\mathfrak{p}_i]$. Denote the Dieudonné module of $\tilde{A}[p]$ over $\overline{\F}_p$ by $\mathcal{D}$. Since $\mathcal{O}_K \otimes \Z_p = \bigoplus_{i=1}^n \mathcal{O}_{K,\mathfrak{p}_i}$, we have that $\mathcal{D}$ also decomposes as a direct sum $\mathcal{D} = \bigoplus_{i=1}^n \mathcal{D}(\mathfrak{p}_i)$, where $\mathcal{D}(\mathfrak{p}_i)$ has dimension $f_i$. For the $a$-number we have that $a(\mathcal{D}) = \sum_{i=1}^n a(\mathcal{D}(\mathfrak{p}_i))$. Recall that the Dieudonné module of a local-local group scheme has strictly positive $a$-number. Note that it is also possible to study these decompositions if $p$ ramifies, see also \citet{goren2012}.

\subsection{Reduction of CM Abelian Surfaces}\label{sec:reduction_cm_surfaces}
Let us now consider the situation where we reduce a CM abelian surface over a number field $k$ at a prime $\mathfrak{P}$ of $k$. We expect a similar result as in the elliptic curve case, where we either find an ordinary or a supersingular elliptic curve $E$ over a finite field, depending on what happens to the $p$-torsion of the reduced curve. Let $E$ be an elliptic curve defined over some number field $k$ with CM by the ring of integers $\mathcal{O}_K$ of an imaginary quadratic field $K$, and let $\mathfrak{P}$ be a prime of $k$. By a theorem of \citet{deuring1941} (see also \citet[Chaper 13, Theorem 12]{lang1987} for a modern statement), we know that the type of the reduction of $E \mod \mathfrak{P}$, denoted by $\tilde{E}$, depends only on the splitting behaviour of $p = \mathfrak{P} \cap \Z$ in $\mathcal{O}_K$: If $p$ splits then $\tilde{E}$ is ordinary, and if $p$ is ramified or inert then $\tilde{E}$ is supersingular.

The following theorem gives the first half of the answer to the problem of generalising Deuring's result to the reduction of CM abelian surfaces.

\begin{theorem}[\cite{goren1997}, Theorem 1]\label{thm:genus_2_reduction_types}
Let $K/\Q$ be a cyclic quartic CM field. Let $A/\Q^\text{alg}$ be an abelian surface. Assume that $A$ has complex multiplication by $\mathcal{O}_K$. Let $\mathfrak{P}$ be a prime of $\Q^\text{alg}$, $\mathfrak{p}_1 = \mathfrak{P} \cap \mathcal{O}_K$, $p = \mathfrak{p}_1 \cap \Z$. Assume that $p$ is unramified in $K$. Then $\tilde{A}$, the reduction of $A \mod \mathfrak{P}$, and $f(\tilde{A})$, $a(\tilde{A})$ are determined by the decomposition of $p$ in $\mathcal{O}_K$ as follows:
\begin{enumerate}
\item If $p = \mathfrak{p}_1 \mathfrak{p}_2 \mathfrak{p}_3 \mathfrak{p}_4$, then $\tilde{A}$ is ordinary and simple, $f(\tilde{A}) = 2$, $a(\tilde{A}) = 0$.
\item If $p = \mathfrak{p}_1 \mathfrak{p}_2$, then $\tilde{A}$ is superspecial, $f(\tilde{A}) = 0$, $a(\tilde{A}) = 2$.
\item If $p = \mathfrak{p}_1$, then $\tilde{A}$ is supersingular non-superspecial, $f(\tilde{A}) = 0$, $a(\tilde{A}) = 1$.
\end{enumerate}
\end{theorem}

We can see that in the cyclic quartic CM field case no possible reduction of an abelian surface induces a $p$-torsion of non-ordinary type, see \Fref{ex:genus_2_p_torsion}. This type only occurs in the second half of the answer, but we will not restate it here since it involves non-Galois quartic CM fields which are not relevant to the rest of this work. We urge the interested reader to consult \citet[Theorem 2]{goren1997} instead.

\section{Cyclic CM Types}\label{sec:cyclic_cm_types}
Let $g\geq 1$. Note that a CM type $\Phi \subset \{\phi_1,\dots,\phi_g,\overline{\phi}_1,\dots,\overline{\phi}_g\}$ can be identified with a binary string $v_\Phi$ of length $g$ such that $v_{\Phi,i} = 0$ if $\Phi_i = \phi_i$ and $v_{\Phi,i} = 1$ if $\Phi_i = \overline{\phi}_i$. Alternatively, we can also consider an extended representation of  $\Phi$ as a binary string of length $2g$ by concatenating $v_\Phi$ and its complement. This does not give any extra information but will be useful. 

Let $B(g)$ be the set of binary strings of length $2g$ with the additional property that entries that are at distance $g$ from each other must be different. For example, $101010\in B(3)$, but $110010\notin B(3)$. As we said, elements in $B(g)$ can be thought of as representing CM types.

In this section we want to study CM types of cyclic CM fields. The cyclic group of order $2g$ acts naturally via cyclic shifts on the set of binary strings of length $2g$. It is easy to see that $B(g)$ is preserved (setwise) by this action. This induces an equivalence relation on $B(g)$ (with the equivalence classes given by the orbits of the group). Let $E(g)$ be the set of equivalence classes. 

The period $k$ of a binary string $b$ is the smallest positive integer such that $b$ is fixed under a cyclic shift of $k$. For example, $1010$ has period $2$, while $1110$ has period $4$. Clearly, the period is preserved under cyclic shifts, so the period of elements of $E(g)$ is also well-defined.

Let $P(k)$ be the number of elements of $B(g)$ of period $k$. Note that, as the notation suggests, this does not actually depend on $g$, as will be shown in the proof of~\Fref{thm:cm_types_count}. We now determine formulas for $P(k)$ and $|E(g)|$. Denote Euler's totient function by $\varphi(x)$, and the Möbius function by $\mu(x)$. 

\begin{theorem}\label{thm:cm_types_count}
If $k,g\geq 1$, then
\begin{align*}
P(k)&=\sum_{\substack{x \mid k \\ x \mathrm{~odd }}} \mu(x)2^{k/2x}\\
|E(g)|&=\frac{1}{2g}\sum_{\substack{d \mid g \\ d \mathrm{~odd }}} \varphi(d)2^{g/d}.
\end{align*}
\end{theorem}

\begin{proof}
First, note that there is a one-to-one correspondence between $B(g)$ and the set of binary strings of length $g$, given by taking the first $g$ bits of a string. (In the opposite direction, simply append the complement of the string of length $g$.) It follows that $|B(g)|=2^g$.

Let $b\in B(g)$ have period $k$. Clearly, $k$ must divide the length of $b$, that is $2g$. On the other hand, $k$ cannot divide $g$, as otherwise the second half of the string would be identical to the first half. Equivalently, $2g/k$ is odd.

Let $k$ be a divisor of $2g$ such that $2g/k$ is odd, and let $b$ be an element of $B(g)$ of period $k$. Clearly, $b$ is completely determined by its first $k$ entries. In fact, the first $k/2$ entries suffice to determine $b$: since $2g/k$ is odd, the first $k/2$ entries of the second half of $b$ must be equal to the second set of $k/2$ entries, which must then be the complement of the first $k/2$ entries.  Conversely, if one arbitrarily chooses the first $k/2$ entries of an element $b$ of $B(g)$, one can complete the rest of the entries using the procedure above to get an element of period dividing $k$. In particular, $P(k)$ does not depend on $g$, as claimed earlier.

Each of the $2^g$ elements of $B(g)$ has some period $k$ for $k$ as above, so
\begin{equation*}
|B(g)|=2^g =\sum_{\substack{k \mid 2g \\ \frac{2g}{k} \mathrm{odd }}} P(k).
\end{equation*}
Writing $m=2g$, we can rewrite this as
\begin{equation*}
2^{m/2} =\sum_{\substack{k \mid m \\ \frac{m}{k} \mathrm{odd }}} P(k).
\end{equation*}

By Möbius inversion, it follows that
\begin{equation*}
P(k)=\sum_{\substack{x \mid k \\ x \mathrm{~odd }}} \mu(x)2^{k/2x},
\end{equation*}
proving our first claim. If $b\in B(g)$ has period $k$, then its equivalence class also has size $k$. So
\begin{align*}
|E(g)| &=\sum_{\substack{k\mid 2g \\ \frac{2g}{k} \mathrm{odd }}}\frac{1}{k} P(k)\\
&=\sum_{\substack{k\mid 2g \\ \frac{2g}{k} \mathrm{odd }}}\frac{1}{k} \sum_{\substack{x\mid k \\ x \mathrm{~odd }}} \mu(x)2^{k/2x}.
\end{align*}
Writing $d=2gx/k$, this becomes
\begin{align*}
|E(g)| &=\frac{1}{2g}\sum_{\substack{d\mid 2g \\ d \mathrm{~odd }}}d2^{g/d} \sum_{x\mid d} \frac{\mu(x)}{x}\\
&=\frac{1}{2g}\sum_{\substack{d\mid g \\ d \mathrm{~odd }}}2^{g/d}\varphi(d).
\end{align*}
\end{proof}

Denote the subset of primitive equivalence classes by $E'(g) \subset E(g)$. Since the primitive CM types correspond to elements of $E(g)$ with period $2g$, \Fref{thm:cm_types_count} immediately implies the following corollary.
\begin{corollary}
If $g \geq 1$, then
\begin{equation*}
|E'(g)|=P(2g)/2g=\frac{1}{2g} \sum_{\substack{x\mid g \\ x \mathrm{~odd }}} \mu(x)2^{g/x}.
\end{equation*}
\end{corollary}

In particular, in the \emph{On-Line Encyclopedia of Integer Sequences} \cite{oeis2022}, $|E(g)|$ is given by the sequence A000016\footnote{https://oeis.org/A000016}, $|E'(g)|$ is given by A000048\footnote{https://oeis.org/A000048}, whereas $|E(g)|-|E'(g)|$ is given by A053734\footnote{https://oeis.org/A053734}. Here is a table of the first few values:

\begin{table}[h!]
\centering
\begin{tabular}{c|c|c|c|c|c|c|c|c|c|c|c|c|c|c}
$g$ & 1 & 2 & 3 & 4 & 5 & 6 & 7 & 8 & 9 & 10 & 11 & 12 & 13 & 14 \\
\hline
Total & 1 & 1 & 2 & 2 & 4 & 6 & 10 & 16 & 30 & 52 & 94 & 172 & 316 & 586 \\
Primitive & 1 & 1 & 1 & 2 & 3 & 5 & 9 & 16 & 28 & 51 & 93 & 170 & 315 & 585 \\
Imprimitive & 0 & 0 & 1 & 0 & 1 & 1 & 1 & 0 & 2 & 1 & 1 & 2 & 1 & 1
\end{tabular}
\caption{Number of equivalence classes of CM types for cyclic CM fields of degree $2g$.}\label{table:cyclic_cm_types}
\end{table}

Next we give representatives for the equivalence classes for some small values of $g$. Let $K$ be a cyclic CM field of degree $2g$. Denote the generator of $\Gal(K/\Q) = C_{2g}$ by $\tau$. Then complex conjugation corresponds to $\tau^g$.

\paragraph{$g = 2$:}
There is one primitive class with representative $\{1, \tau\}$.

\paragraph{$g = 3$:}
There is one primitive class with representative $\{1, \tau, \tau^2\}$, and one imprimitive class with representative $\{1, \tau^2, \tau^4\}$ fixed by $\tau^2$ (generating $C_3$).

\paragraph{$g = 4$:}
There are two primitive classes with representatives $\{1, \tau, \tau^2, \tau^3\}$ and $\{1, \tau, \tau^3, \tau^6\}$.

\paragraph{$g = 5$:}
There are three primitive classes with representatives
\begin{gather*}
\{1, \tau, \tau^2, \tau^3, \tau^4\}, \\
\{1, \tau, \tau^2, \tau^4, \tau^8\}, \\
\{1, \tau, \tau^3, \tau^4, \tau^7\}.
\end{gather*}
There is also one imprimitive class with representative $\{1, \tau^2, \tau^4, \tau^6, \tau^8\}$ fixed by $\tau^2$ (generating $C_5$).

\paragraph{$g = 6$:}
There are five primitive classes with representatives
\begin{gather*}
\{1, \tau, \tau^2, \tau^3, \tau^4, \tau^5 \}, \\
\{1, \tau, \tau^2, \tau^3, \tau^5, \tau^{10} \}, \\
\{1, \tau, \tau^2, \tau^4, \tau^5, \tau^9 \}, \\
\{1, \tau, \tau^2, \tau^5, \tau^9, \tau^{10} \}, \\
\{1, \tau, \tau^3, \tau^5, \tau^8, \tau^{10} \}.
\end{gather*}
There is also one imprimitive class with representative $\{1, \tau, \tau^4, \tau^5, \tau^8, \tau^9 \}$ fixed by $\tau^4$ (generating $C_3$).

\section{Reduction of Higher Dimensional CM Abelian Varieties}\label{sec:higher_dim_reduction}
In this section we consider the situation of 3-dimensional abelian varieties with CM by the ring of integers $\mathcal{O}_K$ of a cyclic sextic CM field $K$. We saw in \Fref{sec:cyclic_cm_types} that for a cyclic sextic CM field there exists a unique primitive CM type. We are interested in the reduction types in this case. Abelian varieties that have CM by the other unique but imprimitive CM type (induced from an imaginary quadratic subfield) can be analysed by reducing to lower dimension via an isogeny, and is thus already solved. Hence, we can safely concentrate on the primitive setting.

\begin{lemma}[Adapted from \cite{goren1997}, Local-Local Lemma]\label{lem:local_local_lemma}
Let $A$ be an abelian variety over some finite field of characteristic $p$. If the Frobenius morphism $F : A \to A^{(p)}$ satisfies $F^m = \alpha p^n$ for some positive integers $m \geq n$, and $\alpha \in \Aut(A)$, then $A[p]$ is a local-local group scheme.
\end{lemma}
\begin{proof}
Denote the \emph{Verschiebung} morphism by $V : A^{(p)} \to A$. Then $VF = p$ in $\End(A)$, and iterating $m$ times, yields $V^mF^m = p^m$. Using the assumption, we have $V^m\alpha p^n = p^m$ and hence $(V^m - p^{m-n}\alpha^{-1})\alpha p^n = 0$. Since $\alpha p^n$ is surjective, $V^m = \alpha^{-1}p^{m-n}$. Now both Frobenius and Verschiebung are nilpotent on $A[p]$ and thus $A[p]$ is a local-local group scheme.
\end{proof}

Consider the same setup as in \Fref{sec:reduction_cm_abvar}: Suppose we are given a principal abelian variety $A$ defined over a number field $k$ with CM by $(K,\Phi)$. Assume here that $k$ is Galois, and let $\mathfrak{P}$ be a prime of $k$ such that $A$ has good reduction at $\mathfrak{P}$. Let $\sigma$ be a Frobenius automorphism in $\Gal(k/K')$ corresponding to $\mathfrak{P}$. Let $\F_q$ be the residue field $\mathcal{O}_K/\mathfrak{P}$.

\begin{lemma}[\citet{ekedahl1987}, Section 2 Proposition 2.4]\label{lem:ssing_product_criterion}
The reduction of $A \mod \mathfrak{P}$ is isomorphic over $\overline{\F}_q$ to a product of supersingular elliptic curves if and only if $\sigma \overline{\Phi} = \Phi$.
\end{lemma}

Now we are ready to state the generalisation of Deuring's result about the reduction of CM elliptic curves, and \Fref{thm:genus_2_reduction_types} in the dimension 2 case. The following theorem decides the reduction types (see \Fref{ex:genus_3_p_torsion}) of 3-dimensional abelian varieties with CM by a cyclic sextic CM field.

\begin{theorem}\label{thm:genus_3_reduction_types}
Let $K$ be a cyclic sextic CM field. Let $A/k$ be a principal abelian variety of dimension 3 over a number field $k$ with CM by $\mathcal{O}_K$ and the unique primitive CM type. Let $\mathfrak{P}$ be a prime of $k$, $\mathfrak{p}_1 = \mathfrak{P} \cap \mathcal{O}_K$, $p = \mathfrak{p}_1 \cap \Z$. Assume that $p$ is unramified in $K$. Then the reduction of $A \mod \mathfrak{P}$, $\tilde{A}$, and $f(\tilde{A})$, $a(\tilde{A})$ are determined by the decomposition of $p$ in $\mathcal{O}_K$ as follows:
\begin{enumerate}
\item If $p = \mathfrak{p}_1 \mathfrak{p}_2 \mathfrak{p}_3 \mathfrak{p}_4 \mathfrak{p}_5 \mathfrak{p}_6$, then $\tilde{A}$ is ordinary, $f(\tilde{A}) = 3$, $a(\tilde{A}) = 0$.
\item If $p = \mathfrak{p}_1 \mathfrak{p}_2 \mathfrak{p}_3$, then $\tilde{A}$ is superspecial, $f(\tilde{A}) = 0$, $a(\tilde{A}) = 3$.
\item If $p = \mathfrak{p}_1 \mathfrak{p}_2$, then $f(\tilde{A}) = 0$, $a(\tilde{A}) = 2$.
\item If $p = \mathfrak{p}_1$, then $f(\tilde{A}) = 0$, $a(\tilde{A}) = 1$.
\end{enumerate}
\end{theorem}

\begin{proof}
By \Fref{lem:abelian_ext_reflex_field}, $K$ is its own reflex field. Hence we have inclusions $k \supset K = K' \supset \Q$.
	Denote the generator of $\Gal(K/\Q) = C_6$ by $\tau$. According to \Fref{sec:cyclic_cm_types}, we have one primitive CM type which, without loss of generality, we may assume is given by $\Phi = \{1,\tau,\tau^2\}$. Complex conjugation corrsponds to the order two element $\tau^3$ and so the conjugate of $\Phi$ is $\overline{\Phi} = \{\tau^3,\tau^4,\tau^5\}$. Finally, the reflex type is $\Phi' = \Phi^{-1} = \{1,\tau^4,\tau^5\}$.

If $p = \prod_{i=1}^\ell \mathfrak{p}_i$ is the decomposition of $p$ in $K$, then $A[p] = \bigoplus_{i=1}^\ell A[\mathfrak{p}_i]$. Let $\F_q$ be the residue field $\mathcal{O}_K/\mathfrak{P}$ and let $\Fr_\mathfrak{P}$ be the Frobenius fixing $\F_q$. By \Fref{thm:shimura_taniyama_formula}, the ideal $(\pi_\mathfrak{P})$ corresponding to $\Fr_\mathfrak{P}$ in $\mathcal{O}_K \hookrightarrow \End(\tilde{A})$ is a power of $N_{\Phi'}(\mathfrak{p}_1)$.

\emph{Type 1.} The prime $p$ splits completely as $p = \mathfrak{p}_1 \mathfrak{p}_2 \mathfrak{p}_3 \mathfrak{p}_4 \mathfrak{p}_5 \mathfrak{p}_6$ and $\tau$ permutes the factors cyclically. We can assume that $\tau(\mathfrak{p}_i) = \mathfrak{p}_{i+1}$. We compute $N_{\Phi'}(\mathfrak{p}_1) = \mathfrak{p}_1\mathfrak{p}_5\mathfrak{p}_6$, and so $\Fr_\mathfrak{P}$ acts trivially on $A[\mathfrak{p}_2] \oplus A[\mathfrak{p}_3] \oplus A[\mathfrak{p}_4]$. Hence the étale part of $\tilde{A}[p]$ has order $p^3$ and so $\tilde{A}$ is ordinary with $p$-rank $f(\tilde{A}) = 3$ and $a$-number $a(\tilde{A}) = 0$.

\emph{Type 2.} The prime $p$ splits as $p = \mathfrak{p}_1 \mathfrak{p}_2 \mathfrak{p}_3$ and we can assume that $\kappa = \tau^2$ (generating the $C_3$ subgroup of $C_6$) acts as $\kappa(\mathfrak{p}_i) = \mathfrak{p}_{i+1}$. Hence, $\tau$ acts as the permutation $(3\,1\,2)$. We compute $N_{\Phi'}(\mathfrak{p}_1) = \mathfrak{p}_1\mathfrak{p}_3\mathfrak{p}_2 = p$. By \Fref{lem:local_local_lemma}, $\tilde{A}[p]$ is a local-local group scheme, and $f(\tilde{A}) = 0$. The decomposition of the Dieudonné module of $\tilde{A}[p]$ is $\mathcal{D} = \mathcal{D}(\mathfrak{p}_1) \oplus \mathcal{D}(\mathfrak{p}_2) \oplus \mathcal{D}(\mathfrak{p}_3)$, and all the factors come from local-local group schemes. Hence, $a(\tilde{A}) = 3$.

\emph{Type 3.} The prime $p$ splits as $p = \mathfrak{p}_1 \mathfrak{p}_2$ and we can assume that $\lambda = \tau^3$ (generating the $C_2$ subgroup of $C_6$) acts as $\lambda(\mathfrak{p}_i) = \mathfrak{p}_{i+1}$. Hence, $\tau$ also acts as $\tau(\mathfrak{p}_i) = \mathfrak{p}_{i+1}$. We compute $N_{\Phi'}(\mathfrak{p}_1) = \mathfrak{p}_1\mathfrak{p}_1\mathfrak{p}_2 = p\mathfrak{p}_1$ and so $f(\tilde{A}) = 0$. Let $\sigma$ be the Frobenius automorphism in $\Gal(k/K')$ corresponding to $\mathfrak{P}$. The decomposition group of $\mathfrak{P}$ is $C_3 < C_6$ and so $\sigma$ is equal to one of $\{\tau^2, \tau^4\}$. In any case, $\sigma \overline{\Phi} \neq \Phi$, and so by \Fref{lem:ssing_product_criterion} $\tilde{A}$ is not isomorphic to a product of supersingular elliptic curves. Hence, $a(\tilde{A}) = 2$.

\emph{Type 4.} Here we simply have $A[p] = A[\mathfrak{p}_1]$. The Frobenius automorphism $\sigma$ corresponding to $\mathfrak{P}$ is equal to one of $\{\tau,\tau^5\}$ (since the decomposition group of $\mathfrak{P}$ is the whole $C_6$). Again, $\sigma \overline{\Phi} \neq \Phi$ and so $\tilde{A}$ is not isomorphic to a product of supersingular elliptic curves. Thus $f(\tilde{A}) = 0$ and $a(\tilde{A}) = 1$.
\end{proof}

Using \Fref{lem:ssing_product_criterion} we can also decide between ordinary and superspecial reduction in arbitrary dimension, at least for principal abelian varieties with CM by a cyclic CM field and primitive CM type.

\begin{theorem}\label{thm:genus_g_reduction_types}
Let $K$ be a cyclic CM field of degree $2g$. Let $A/k$ be a principal abelian variety of dimension g over a number field $k$ with  complex multiplication by $\mathcal{O}_K$ such that the CM type is primitive. Let $\mathfrak{P}$ be a prime of $k$, $\mathfrak{p}_1 = \mathfrak{P} \cap \mathcal{O}_K$, $p = \mathfrak{p}_1 \cap \Z$. Assume that $p$ is unramified in $K$. Then the reduction of $A \mod \mathfrak{P}$, $\tilde{A}$, and $f(\tilde{A})$, $a(\tilde{A})$ are determined by the decomposition of $p$ in $\mathcal{O}_K$ as follows:
\begin{enumerate}
\item If $p = \mathfrak{p}_1 \cdots \mathfrak{p}_{2g}$, then $\tilde{A}$ is ordinary, $f(\tilde{A}) = g$, $a(\tilde{A}) = 0$.
\item If $p = \mathfrak{p}_1 \cdots \mathfrak{p}_g$, then $\tilde{A}$ is superspecial, $f(\tilde{A}) = 0$, $a(\tilde{A}) = g$.
\end{enumerate}
\end{theorem}

\begin{proof}
Again, $K$ is its own reflex field and $k \supset K = K' \supset \Q$. Fix a generator $\tau$ of $\Gal(K/\Q) = C_{2g}$. Assume the CM type is given by $\Phi = \{\tau^{i_1},\tau^{i_2},\dots,\tau^{i_g}\}$ with distinct exponents $i_m \neq i_n$ for $m \neq n$. Without loss of generality, we can always write it as $\Phi = \{1,\tau^{j_2},\dots,\tau^{j_g}\}$ (again with distinct exponents) by multiplying through by $\tau^{-i_1}$ and so passing to an equivalent type. Complex conjugation corresponds to $\tau^g$ and generates $C_2 < C_{2g}$, and so $\overline{\Phi} = \{\tau^g,\tau^{j_2 + g},\dots,\tau^{j_g + g}\}$. For the reflex type we have $\Phi' = \Phi^{-1} = \{1,\tau^{2g-j_2},\dots,\tau^{2g-j_g}\}$.

If $p = \prod_{i=1}^\ell \mathfrak{p}_i$ is the decomposition of $p$ in $K$, then $A[p] = \bigoplus_{i=1}^\ell A[\mathfrak{p}_i]$. Let $\F_q$ be the residue field $\mathcal{O}_K/\mathfrak{P}$ and let $\Fr_\mathfrak{P}$ be the Frobenius fixing $\F_q$. By \Fref{thm:shimura_taniyama_formula}, the ideal $(\pi_\mathfrak{P})$ corresponding to $\Fr_\mathfrak{P}$ in $\mathcal{O}_K \hookrightarrow \End(\tilde{A})$ is a power of $N_{\Phi'}(\mathfrak{p}_1)$.

\emph{Type 1.} Here $\ell = 2g$ and the prime $p$ splits completely as $p = \mathfrak{p}_1 \cdots \mathfrak{p}_{2g}$ and $\tau$ permutes the factors cyclically. We can assume that $\tau(\mathfrak{p}_i) = \mathfrak{p}_{i+1}$. We compute $N_{\Phi'}(\mathfrak{p}_1) = \mathfrak{p}_1 \mathfrak{p}_{2g-j_2+1} \cdots \mathfrak{p}_{2g-j_g+1}$. Let $Q = \{\mathfrak{p}_1\} \cup \{\mathfrak{p}_2, \dots, \mathfrak{p}_{2g}\} \cap \{\mathfrak{p}_{2g-j_2+1}, \dots, \mathfrak{p}_{2g-j_g+1}\}$, a set of size $g$. Then $\Fr_\mathfrak{P}$ acts trivially on $\bigoplus_{\mathfrak{q} \in Q} A[\mathfrak{q}]$. Hence the étale part of $\tilde{A}[p]$ has order $p^g$ and so $\tilde{A}$ is ordinary with $p$-rank $f(\tilde{A}) = g$ and $a$-number $a(\tilde{A}) = 0$.

\emph{Type 2.} Here $\ell = g$ and the prime $p$ splits as $p = \mathfrak{p}_1 \cdots \mathfrak{p}_g$. The decomposition group of $p$ is equal to $C_2 < C_{2g}$ and so the Frobenius automorphism $\sigma$ in $\Gal(k/K')$ corresponding to $\mathcal{P}$ is equal to complex conjugation $\tau^g$. Hence, $\sigma\overline{\Phi} = \Phi$ and so \Fref{lem:ssing_product_criterion} implies that $\tilde{A}$ is isomorphic to a product of supersingular elliptic curves. In other words, $\tilde{A}$ is superspecial and so $f(\tilde{A}) = 0$, $a(\tilde{A}) = 3$.
\end{proof}

\begin{remark}
Of course \Fref{thm:genus_g_reduction_types} immediately implies the first two cases of \Fref{thm:genus_3_reduction_types}, but we find it useful to have a self-contained proof in that case since it give an example of applying the decomposition theory of Dieudonné modules to reason about $a$-numbers.
\end{remark}

\section{Constructing Supersingular Genus 2 Curves}\label{sec:genus_2_construction}
Recall that in dimension 1, for elliptic curves, the \emph{Hilbert class polynomial} $H_K \in \Z[X]$ of an imaginary quadratic number field K has as roots the $j$-invariants of elliptic curves with CM by the ring of integers $\mathcal{O}_K$ of $K$. The roots also generate the \emph{Hilbert class field} of $K$, the maximal abelian unramified extension of $K$. The \emph{CM method} for elliptic curves uses the reduction of $H_K \pmod{p}$ for special primes $p$ to construct elliptic curves defined over $\F_p$ with a prescribed number of points. In \Fref{sec:reduction_cm_surfaces} we explained that it is also possible to control the type of the resulting elliptic curve. It is either ordinary or supersingular, depending on whether $p$ splits or is inert in $\mathcal{O}_K$, respectively. \citet{broeker2009} used this idea to give an algorithm to construct supersingular elliptic curves.

The generalisation of the CM method to dimension 2, for abelian surfaces, follows from the \emph{Igusa class polynomials} $H_{K,i}$ for $i \in \{1,2,3\}$ of a quartic CM field $K$. The roots of these three polynomials are the Igusa invariants of all genus 2 curves whose Jacobian has CM by the ring of integers of the quartic field $K$. The Igusa class polynomials can be explicitly calculated by computing the Igusa invariants of all possible abelian surfaces with CM by $\mathcal{O}_K$ following the construction described in \Fref{sec:cm_explicit_construction}, see \citet{vanwamelen1999}, \citet{weng2003}, and \citet{streng2014}.

The roots again generate class fields, and the reduction of $H_{K,i} \pmod{p}$ can be used to construct genus 2 curves whose Jacobian has a fixed number of points. Previous work considered only two reduction types of these curves: Ordinary and non-ordinary genus 2 curves, see \citet{weng2003}, and \citet{hitt2011}. Recall \Fref{ex:genus_2_p_torsion} for the $p$-torsion structure of those types.

Let us quickly recall the structure theory of quartic CM fields. They arise in one of three possible structure types:
\begin{itemize}
\item $K/\Q$ is Galois with $\Gal(K/\Q) = C_2 \times C_2$, each CM type is non-primitive and induced from an imaginary quadratic subfield.
\item $K/\Q$ is cyclic Galois with $\Gal(K/\Q) = C_4$, all four CM types are equivalent and primitive, see \Fref{sec:cyclic_cm_types}.
\item $K/\Q$ is non-Galois. Denote its normal closure be $L$, then $\Gal(L/\Q) = D_4$, and each CM type is primitive. There are two equivalence classes of CM types, generated by $(\phi_1,\phi_2)$ and $(\phi_1,\overline{\phi_2})$, respectively.
\end{itemize}

We may put aside the first case of a \emph{biquadratic} quartic CM field, since the only CM type is non-primitive and hence every abelian surface with CM by such a field would be non-simple. This reduces us to the case of elliptic curves, for which the CM method is already described in terms of the reduction of the corresponding Hilbert class fields. Non-ordinary reduction is only possible in the case of a non-Galois quartic CM field. Even though for such fields all possible reduction types as in \Fref{ex:genus_2_p_torsion} occur (see \citet[Theorem 2]{goren1997}), we will deal with the simpler cyclic Galois case which is sufficient for our purpose.

\Fref{thm:genus_2_reduction_types} motivates \Fref{alg:cm_method_genus_2} for constructing supersingular non-superspecial, and superspecial genus 2 curves.

\begin{algorithm}[H]
\caption{}
\begin{algorithmic}[1]
\Require A cyclic quartic CM field K and a type: supersingular non-superspecial, or superspecial.
\Ensure A prime $p$ and a genus-2 curve $C$ defined over (possibly an extension of) $\F_p$ such that $\Jac(C)$ has the required type.
\State Construct a genus 2 curve $C$ which has CM by the ring of integers $\mathcal{O}_K$ of $K$.
\State If the type is superspecial, then find a prime $p$ which splits as $p = \mathfrak{p}_1\mathfrak{p}_2$ in $K$. Otherwise, if the type is supersingular non-superspecial, then find a prime $p$ which remains inert in $K$.
\State Output the tuple $(p, C \mod p)$.
\end{algorithmic}
\label{alg:cm_method_genus_2}
\end{algorithm}

A cyclic field has infinitely many inert primes and hence the runtime of \Fref{alg:cm_method_genus_2} is dominated by constructing the genus 2 curve $C$ in step 1. If the CM field $K$ used as input has known Igusa class polynomials, then the algorithm runs in expected time $O(1)$. For a fixed cyclic CM field we can determine the splitting behaviour of unramified primes using the following lemma.

\begin{lemma}\label{lem:cyclic_cm_prime_splitting}
Let $K$ be a cyclic CM field. Let $f$ be the smallest integer such that $K \subset \Q(\zeta_f)$ is a subfield of the cyclotomic field $\Q(\zeta_f)$ (i.e. $f$ is the finite part of the conductor $\mathfrak{f}$ of $K$). Let $p$ be a prime unramified in $K$. Then the decomposition type of $p$ in $K$ depends only on the residue of $p \pmod{f}$.
\end{lemma}
\begin{proof}
We give a somewhat constructive proof. The cyclic CM field $K$ is a subfield of $\Q(\zeta_f)$ by the Kronecker--Weber theorem (\citet[Chapter V Theorem 1.10]{neukirch1999}), and $K$ and $\Q(\zeta_f)$ have the same ramified primes by the conductor-discriminant formula (\citet[Chapter VII Corollary 10.8]{neukirch1999}). By \citet[Chapter III Lemma 1.3]{casselsfroehlich1976} the usual inertia degree $f_p$ of $p$ in $\Q(\zeta_f)$ is also equal to the least integer $f_p \geq 1$ such that $p^{f_p} \equiv 1 \pmod f$. Since $p$ is unramified we may assume that it splits in $\Q(\zeta_f)$ as $p = \mathfrak{p}_1 \cdots \mathfrak{p}_g$ for some integer $g$. We have that $[\Q(\zeta_f):\Q] = \varphi(f)$ (where $\varphi(f)$ is Euler's totient function) for cyclotomic fields, and also $[\Q(\zeta_f):\Q] = gf_p$ from the decomposition of $p$. Hence, the assertion for $\Q(\zeta_f)$ follows from the formula $g = \varphi(f)/f_p$. Now let $H < (\Z/f\Z)^\times$ correspond to the subgroup $\Gal(\Q(\zeta_f)/K) < \Gal(\Q(\zeta_f)/\Q)$ fixing $K$. Then the inertia degree $f'_p$ of $p$ in $K$ is equal to the least integer $f'_p \geq 1$ such that $p^{f'_p} \equiv h \pmod{f}$ for $h \in H$. Assume $[\Q(\zeta_f):K] = d$, then $[K:\Q] = \varphi(f)/d$. Hence, if $p = \mathfrak{p}_1 \cdots \mathfrak{p}_{g'}$ in $K$ then $g' = \varphi(f)/(f'_p d)$. Thus the assertion follows for $K$.
\end{proof}

Fix now a cyclic quartic CM field $K$ and let $f$ as in \Fref{lem:cyclic_cm_prime_splitting}, then using the ideas of the proof of \Fref{lem:cyclic_cm_prime_splitting} we can determine the possible residue classes modulo $f$ that correspond to each prime splitting type in $K$. These can then be used in step 2 of \Fref{alg:cm_method_genus_2} instead of trial decomposition.

Using \Fref{lem:cyclic_cm_prime_splitting} works for general cyclic CM fields of degree $2g$, but in the case that $g=2$ we have a much more direct characterisation if we are only interested in superspecial non-supersingular primes of reduction. This follows from the following observation. Since for a cyclic quartic field there are only three splitting types, and two of them are determined by an even number of factors, we can alternatively use the following theorem of Stickelberger to pin down a prime of supersingular non-superspecial reduction.

\begin{theorem}[Stickelberger, see \citet{carlitz1953}]\label{thm:stickelberger_splitting}
Let $K$ be a number field of degree $n$ with discriminant $D$. If $p$ is an unramified prime in $K$, then the number $m$ of primes over $p$ satisfies
\begin{equation*}
\left(\frac{D}{p}\right) = (-1)^{n-m},
\end{equation*}
where $(D/p)$ denotes the Kronecker symbol.
\end{theorem}

As we said above, using \Fref{thm:stickelberger_splitting} we obtain an explicit version of \Fref{alg:cm_method_genus_2} to generate supersingular non-superspecial genus 2 curves over $\F_p$, where we can also conveniently control the size of $p$.

\begin{algorithm}[H]
\caption{}
\begin{algorithmic}[1]
\Require A cyclic quartic CM field K and a positive integer $n$.
\Ensure A prime $p$ of size $p \sim 2^n$ and a genus-2 curve $C$ defined over (possibly an extension of) $\F_p$ such that $\Jac(C)$ is supersingular non-superspecial.
\State Construct a genus 2 curve $C$ which has CM by the ring of integers $\mathcal{O}_K$ of $K$. Let $D$ be the discriminant of $K$.
\State Find a prime $p$ of size $p \sim 2^n$ in $K$ such that $(D/p) = -1$.
\State Output the tuple $(p, C \mod p)$.
\end{algorithmic}
\label{alg:cm_method_genus_2_ssing_non_sspec}
\end{algorithm}

\subsection{Alternate Construction with Base Supersingular Genus 2 Curve}\label{sec:alternate_genus_2_construction}
Superspecial abelian surfaces form a subset of supersingular abelian surfaces. In fact, they belong to the same isogeny class by definition. However, if we were to restrict to separable isogenies only, then superspecial surfaces form an isogeny class of their own that is not connected to any supersingular non-superspecial surface.
\begin{lemma}[\cite{oort1975}, page 36]\label{lem:sep_isogeny_sspec_graph}
Let $A$ be an abelian variety of characteristic $p$ and of dimension 2, and let $\phi:E^g\to A$ be an isogeny of degree $d$ for some supersingular alliptic curve $E$. 
If $\phi$ is separable (i.e. $p\nmid d$) then $A$ is superspecial (i.e. $A\cong E^g$).
\end{lemma}
We can utilise this fact to construct a supersingular non-superspecial surface as an alternative to \Fref{alg:cm_method_genus_2_ssing_non_sspec}.

To produce a base supersingular non-superspecial surface to begin our random walk, we turn to the next proposition, which is a special case of \Fref{thm:genus_2_reduction_types} and \Fref{lem:cyclic_cm_prime_splitting}. It states that the $p$-rank and $a$-number of the fixed hyperelliptic curve $y^2 = x^5 - 1$ over a finite field $k$ depends only on the residue of the characteristic of $k$ modulo 5.
\begin{proposition}[\cite{ibukiyamakatsuraoort1986}, Proposition 1.13]\label{prop:cyclo_5_reduction}
Let $H : y^2 = x^5 - 1$ be a hyperelliptic curve over $\F_{p^r}$, where $p\neq5$, then $H$ is
\begin{enumerate}
\item ordinary if $p\equiv 1 \pmod{5}$,
\item supersingular non-superspecial if $p\equiv 2, 3 \pmod{5}$,
\item superspecial if $p\equiv 4 \pmod{5}$.
\end{enumerate}
\end{proposition}

\begin{proof}
The Jacobian of the curve $y^2 = x^5 - 1$ has CM by the ring of integers $\Z[\zeta_5]$ of the cyclotomic field $\Q(\zeta_5)$. The statement then follows immediately from \Fref{thm:genus_2_reduction_types} and \Fref{lem:cyclic_cm_prime_splitting}.
\end{proof}

In conjunction with \Fref{lem:sep_isogeny_sspec_graph}, we can generate supersingular non-superspecial abelian surfaces for $p \equiv 2,3 \pmod{5}$ by taking random walks from the base hyperelliptic curve in \Fref{prop:cyclo_5_reduction} where the steps in the random walk are $(\ell,\ell)$-isogenies with $\ell$ coprime to $p$. For practical applications and in characteristic different from 2, one would use $(2,2)$-isogenies that can be efficiently computed, see \citet{richelot1837}.

\subsection{Examples}\label{sec:genus_2_examples}
We can find a list of quartic CM fields in the genus 2 section of the \emph{Echidna}\footnote{See \url{https://www.i2m.univ-amu.fr/perso/david.kohel/dbs/index.html}.} databases. It lists quartic CM fields $K = Q[x]/(x^4 + Ax^2 + B)$ in the form of tuples $(D,A,B)$ where $D$ is the discriminant of the totally real subfield $K_0 \subset K$. Hence, $A^2 - 4B = m^2D$ for some $m$.

Let us for example consider the cyclic quartic CM field $K$ given by the tuple $[5, 65, 845]$. Hence, the discriminant of the totally real subfield is $D_{K_0} = 5$ and for the discriminant of $K$ we compute $D_K = 5^3 \cdot 13^2$. 

We find that the integer part of the conductor of $K$ is $f = 5 \cdot 13 = 65$, and using \Fref{lem:cyclic_cm_prime_splitting} the we can compute the splitting behaviour of unramified primes $p$ in $K$ based on their residue $p \mod 65$. We find the following rule:
\begin{equation*}
p = \begin{cases}
\mathfrak{p}_1, \text{ if } p \mod 65 \in \{\pm 2, \pm 3, \pm 7, \pm 8, \pm 12, \pm 17, \pm 18, \pm 22, \pm 23, \pm 27, \pm 28, \pm 32\}, \\
\mathfrak{p}_1 \mathfrak{p}_2, \text{ if } p \mod 65 \in \{-24, -19, -16, -1, 4, 6, 9, 11, 14, 21, 29, 31\}, \\
\mathfrak{p}_1 \mathfrak{p}_2 \mathfrak{p}_3 \mathfrak{p}_4, \text{ if } p \mod 65 \in \{-31, -29, -21, -14, -11, -9, -6, -4, 1, 16, 19, 24\}.
\end{cases}
\end{equation*}

The class group of $\mathcal{O}_K$ is cyclic of order 2 and using the ideas of \Fref{sec:cm_explicit_construction} or the Igusa class polynomials of $K$ we can explicitly construct two genus 2 curves with CM by $\mathcal{O}_K$:
\begin{align*}
C_1 : y^2 &= -1331 x^6 + 1936 x^5 + 6160 x^4 - 4760 x^3 - 8800 x^2 - 748 x - 552, \\
C_2 : y^2 &= -103259 x^6 - 29744 x^5 - 348400 x^3 + 293172 x - 79888.
\end{align*}

Consider for example $p = 2^{128} + 51$ which satisfies $p \equiv -18 \mod 65$. Then the reduced curves $\tilde{C}_1 = C_1 \mod p$ and $\tilde{C}_2 = C_2 \mod p$ are supersingular non-superspecial and defined over $\F_p$. Both curves are isogenous to a product of supersingular elliptic curves and so their endomorphism algebras are isomorphic to the matrix algebra $M_2(B_{p,\infty})$ of $2 \times 2$ matrices over the definite quaternion algebra $B_{p,\infty}$ ramified at $p$. Their endomorphism rings are maximal orders in that algebra, and the ring of integers $\mathcal{O}_K$ of the CM field $K$ injects into $\End(\tilde{C}_1)$ and $\End(\tilde{C}_2)$.

Note that the isogeny $\phi: \Jac(\tilde{C}_i) \to E \times E$ has to be inseparable, as separable isogenies leave the $p$-torsion group scheme invariant. Denote by $E[\Fr_p]$ the kernel of $\Fr_p : E \to E^{(p)}$. Then we have that $\Jac(\tilde{C}_i) = (E \times E)/\psi_{a,b}(\alpha_p)$ via some appropriate embedding of the form $\psi_{a,b} : \alpha_p \hookrightarrow E \times E$ induced by the two isomorphisms $\alpha_p \overset{a}{\to} E[\Fr_p] \overset{b^{-1}}{\to} \alpha_p$
such that $a/b \in \overline{\F}_p \cong \End(\alpha_p)$ and $a/b \neq \F_{p^2}$, see \citet{oort1975}.

\section{Constructing Higher Genus Curves with $p$-rank Zero}\label{sec:higher_genus_construction}
Unlike in genus 2, not every genus 3 curve is hyperelliptic. We restrict ourselves to the 5-dimensional hyperelliptic locus $\mathcal{H}_3$ of the 6-dimensional moduli space $\mathcal{M}_3$ of genus-3 curves. The genus 3 analogue of the Igusa class polynomials as described in \Fref{sec:genus_2_construction} are the (hyperelliptic) \emph{Shioda class polynomials} $H_{K,i}$ for $i \in \{1,2,3,4,5\}$ of a sextic CM field $K$. As with the Igusa class polynomials in genus 2, the Shioda class polynomials can be explicitly calculated by computing the Shioda invariants of all possible 3-dimensional abelian varieties coming from hyperelliptic genus 3 curves with CM by $\mathcal{O}_K$ following the construction described in \Fref{sec:cm_explicit_construction}, see \citet{balakrishnan2016}, \citet{ionica2019}, and \citet{dina2020}. \Fref{thm:genus_3_reduction_types} the motivates an algorithm analogous to \Fref{alg:cm_method_genus_2} for constructing hyperelliptic genus 3 curves with $p$-rank zero and prescribed $a$-number.

\subsection{Complex Multiplication by Cyclotomic Integers}
A different generalisation (also related to the construction of \Fref{sec:alternate_genus_2_construction}) to higher genus can be found in \citet[Section IV.15.4 Example 2)]{shta1961}: Fix a prime $p$ and consider the hyperelliptic curve $C: y^2 = x^\ell - 1$ whose Jacobian $A = \Jac(C)$ has CM by the ring of integers $\Z[\zeta_p]$ of the $\ell$-th cyclotomic field $K = \Q(\zeta_\ell)$. Note that the right hand side $x^\ell - 1$ of the curve equation factors as $x^\ell - 1 = (x - 1)\Phi_\ell(x)$ for the $\ell$-th cyclotomic polynomial $\Phi_\ell(x)$. We immediately see that the genus of $C$ is $g(C) = (\ell-1)/2$ and so the dimension of $A$ is also $(\ell-1)/2$. Again, by \Fref{lem:cyclic_cm_prime_splitting} and the discission in \Fref{sec:reduction_cm_abvar}, the reduction type of $C \mod \mathfrak{P}$ for some prime $\mathfrak{P}$ of $K$ lying above some prime $p = \mathfrak{P} \cap \Z$ depends only on the residue of $p \pmod{\ell}$.

\subsection{Examples}
\citet{weng2001} has computed a few examples of hyperelliptic genus 3 curves with complex multiplication. We consider the totally real field $K_0 = \Q[y]/(y^3 - y^2 - 2y + 1)$ and a quadratic imaginary extension $K = K_0(x)/(x^2 + 1)$ such that $K$ contains $\Q(i)$. Then the curve
\begin{equation*}
C : y^2 = x^7 + 7x^5 + 14x^3 + 7x
\end{equation*}
has CM by $\mathcal{O}_K$.

Let $k$ be some finite field with $\fchar(k) \neq 2$ and consider a genus $g$ hyperelliptic curve defined by an equation $y^2 = f(x)$ for $f(x) \in k[x]$ of degree $2g + 1$ or $2g + 2$. Let $c_i$ denote the coefficient of $x^i$ in the expansion of $f(x)^{(p-1)/2}$, and define for $\ell=0,\dots,g-1$ the $g \times g$ matrix $A_\ell$ with entries $(A_\ell)_{i,j} = (c_{ip-j})^{p^\ell}$. Following \citet{yui1978} and heeding \citet{achterhowe2019} we define the matrix $M = A_{g-1} \cdots A_1 A_0$, and have the following lemma.
\begin{lemma}\label{lem:prank_anumber}
Let $H$ be a genus g hyperelliptic curve defined by $y^2 = f(x)$. Let $A_0$ and $M$ be the matrices as defined above.
\begin{enumerate}
\item The $p$-rank of $H$ is $f(H) = \rank(M)$.
\item The $a$-number of $H$ is $a(H) = g - \rank(A_0)$.
\end{enumerate}
\end{lemma}

Using \Fref{lem:prank_anumber} we can explicitly compute the $p$-rank and $a$-number of the reduction $\tilde{C} = C \mod p$ for various small primes and compare the results to \Fref{thm:genus_3_reduction_types}.

\paragraph{$p = 13$:}
The prime splits as $(13) = \mathfrak{p}_1 \cdots \mathfrak{p}_6$. As expected, we compute the $p$-rank $f(\tilde{C}) = 3$ and $a$-number $a(\tilde{C}) = 0$. Hence the curve $\tilde{C}$ is ordinary. For the $L$-polynomial we find $L_{\tilde{C}}(T) = 2197 T^6 + 676 T^5 + 91 T^4 + 40 T^3 + 7 T^2 + 4 T + 1$, which has Newton slopes $(0,1)$.

\paragraph{$p = 43$:}
The prime splits as $(43) = \mathfrak{p}_1 \mathfrak{p}_2 \mathfrak{p}_3$. In this case, we compute the $p$-rank $f(\tilde{C}) = 0$ and $a$-number $a(\tilde{C}) = 3$. Hence the curve $\tilde{C}$ is superspecial and isomorphic to a product of supersingular elliptic curves. The $L$-polynomial ends up factoring as $L_{\tilde{C}}(T) = (43 T^2 + 1)^3$, which has Newton slopes all $1/2$.

\paragraph{$p = 17$:}
The prime splits as $(17) = \mathfrak{p}_1 \mathfrak{p}_2$. For this example we compute the $p$-rank $f(\tilde{C}) = 0$ and $a$-number $a(\tilde{C}) = 2$. Hence the curve $\tilde{C}$ lies in the codimension 4 or codimension 5 stratum of the moduli space. For the $L$-polynomial we find $L_{\tilde{C}}(T) = 4913 T^6 - 136 T^3 + 1$, which has Newton slopes $(1/3, 2/3)$. Hence $\tilde{C}$ cannot lie in the codimension 5 stratum and so the Jacobian $\Jac(\tilde{C})$ has a $p$-torsion group scheme isomorphic to $I_{3,2}$.

\paragraph{$p = 11$:}
The prime $p = 11$ is inert in $K$. As expected, we compute the $p$-rank $f(\tilde{C}) = 0$ and $a$-number $a(\tilde{C}) = 1$. Hence the curve $\tilde{C}$ lies in the codimension 3 stratum of the moduli space. The Jacobian $\Jac(\tilde{C})$ has a $p$-torsion group scheme isomorphic to $I_{3,1}$. For the $L$-polynomial we find $L_{\tilde{C}}(T) = (11 T^2 + 1)(121 T^4 - 11 T^2 + 1)$, which has Newton slopes all $1/2$. In \Fref{ex:genus_3_p_torsion} we said that a $p$-torsion group scheme $I_{3,1}$ usually has Newton slopes $(1/3,2/3)$, but can also have slopes all $1/2$. Hence, this curve $\tilde{C}$ is one of these outliers. Looking at how the $L$-polynomial factors, we can deduce that $\Jac(\tilde{C})$ is isogenous to $E \times A$ for some supersingular elliptic curve and some supersingular abelian surface $A$. Note that this has to be an inseparable isogeny, similar to the case of supersingular non-superspecial abelian surfaces that we discussed in \Fref{sec:genus_2_examples}.

\section{Endomorphisms of Supersingular Abelian Surfaces}\label{sec:endomorphisms}
Let $A$ be an abelian surface, then we can infer the structure of the endomorphism algebra of $A$ from \emph{Poincaré's reducibility theorem}.
\begin{theorem}[\citet{MumfordAV}, page 174, Corollary 2]
For $X$ simple, the ring $\End_\Q(X)$ is a division ring.
For any abelian variety $X$, if $X=X_1^{n_1}\times\dots\times X_k^{n_k}$, with $X_i$ simple and not isogenous, and $D_i=\End_\Q(X_i)$, then
\begin{equation*}
\End_\Q(X) = M_{n_1}(D_1) \oplus \dots \oplus M_{n_k}(D_k),
\end{equation*}
where $M_k(R)$ is the ring of $k\times k$ matrices over $R$.
\end{theorem}

More explicitly, let $E_1$ and $E_2$ be elliptic curves. Then, for a reducible (i.e. non-simple) abelian surface given by $E_1\times E_2$,
\begin{equation*}
\End_\Q(E_1\times E_2) = \begin{pmatrix}
\End_\Q(E_1) & \Hom_\Q(E_2,E_1)\\
\Hom_\Q(E_1,E_2) & \End_\Q(E_2)
\end{pmatrix}.
\end{equation*}
In the case where $E_1$ and $E_2$ are not isogenous, we have that $\End_\Q(E_1\times E_2) = \End_\Q(E_1) \times \End_\Q(E_2)$.

Let us now consider the situation of supersingular abelian surfaces. For this, let $p$ be a prime and let $E$ be a supersingular elliptic curve over $\overline{\F}_p$. Then we have that the endomorphism algebra $\End_\Q(E)=\End(E) \otimes \Q = B_{p,\infty}$, where $B_{p,\infty}$ is the definite quaternion algebra ramified at $p$. Recall that a superspecial abelian surface is isomorphic to a power of a supersingular elliptic curve, whereas a supersingular non-superspecial abelian surface is only isogenous to such a power (via an inseparable isogeny as we have noted in \Fref{sec:genus_2_examples}), see \citet{oort1975}. Hence, the endomorphism algebra of a supersingular abelian surface is given by $M_2(B_{p,\infty})$, a central simple algebra over $\Q$ of dimension 16 and degree 4.

Recall that the endomorphism ring $\End(E)$ of a supersingular elliptic curve $E$ is a maximal order in $B_{p,\infty}$. Now we show that a similar statement holds for supersingular abelian surfaces: The endomorphism ring of a supersingular abelian surface over $\overline{\F}_p$ is a maximal order in $M_2(B_{p,\infty})$.

\begin{theorem}\label{Thm:EndRingsAreMaxOrders}
Let $A/\overline{\F}_p$ be a supersingular abelian surface, then $\End(A)$ is a maximal order in $M_2(B_{p,\infty})$.
\end{theorem}
\begin{proof}
Let $\mathcal{O} = \End(A) \subseteq \End_\Q(A)$ and consider $n>0$ such that $p \nmid n$. Then there exists an isomorphism of abelian groups
\begin{equation*}
A[n]\cong (\Z/n\Z)^4
\end{equation*}
by \cite[page 64, Proposition 1(3)]{MumfordAV}. The endomorphism ring of this abelian group is given by $M_4(\Z/n\Z)$ and we will denote it by $\End(A[n])$.

We claim that $\mathcal{O}/n\mathcal{O}\cong \End(A[n])$. Firstly, consider the map $\mathcal{O}/n\mathcal{O}\to \End(A[n])$. Our first task is to prove that this map is injective, i.e. that $A[n]$ is a faithful module over $\mathcal{O}/n\mathcal{O}$.

Let $\phi\in\mathcal{O}$ such that $\phi$ annihilates $A[n]$, then one can factor out the multiplication-by-$[n]$ map \cite[page 111, Theorem 1]{MumfordAV}. Hence, there exists some $\psi\in\mathcal{O}$ such that $\phi=[n]\circ\psi$, and so $\phi\in \mathcal{O}/n\mathcal{O}$. Now, we have that $\#A[n]=n^{16}$. Since $\End_\Q(A)$ is preserved under isogenies and we have that $\dim\End(E_1\times E_2)=16$, we have that $\#(\mathcal{O}/n\mathcal{O})=n^{16}$, and so $\mathcal{O}/n\mathcal{O}\cong\End(A[n])$.

Now, because $\mathcal{O}$ is a free $\Z$-module, we have that
\begin{align*}
\O\otimes_{\Z} \Z_{\ell}  &=    \O\otimes_{\Z}\varprojlim \Z/\ell^n\Z &&\text{By definition.}\\
            &\cong  \varprojlim \O\otimes_{\Z}\Z/\ell^n\Z &&\text{$\O$ finitely presented, $\Z/\ell^n\Z$ finite length.\footnotemark}\\
            &\cong  \varprojlim \O/\ell^n\O         &&\text{Commutivity between tensor and quotient.}\\
            &\cong  \varprojlim\End(E[\ell^n])       &&\text{Preceding discussion.}\\
            &=    \End_{\Z_{\ell}}(E[\ell^{\infty}])   &&\\
            &\cong  M_2(\Z_{\ell})              &&
\end{align*}
This gives us that $\mathcal{O}_{\ell}\cong M_4(\Z_{\ell})$ as a $\Z_{\ell}$-algebra, and so $\mathcal{O}_{\ell}$ is maximal. \footnotetext{Note that $\O$ is finitely generated as a $\Z$-module, and $\Z/\ell^n\Z$ has finite length (equal to $n$).}

Now all that is left is to show that $\mathcal{O}$ is maximal at the last remaining place $p$. To do so, we use \cite[Proposition 13.3.4]{VoightQuatBook} which states that the valuation ring consisting of all elements of $\End_\Q(A) \otimes \Q_p$ that are integral over $\Z_p$ is the unique maximal $\Z_p$-order in $\End_\Q(A) \otimes \Q_p$. So we need to prove that $\mathcal{O}_{p}$ is a valuation ring and that $\End_\Q(A) \otimes \Q_p$ is a central simple division algebra.

Let $\phi\in \mathcal{O}$, and denote the \emph{inseparable degree} of $\phi$ by $\deg_i \phi$. Then $q\mid \deg_i\phi$ if and only if $\phi=\pi^q\circ\psi$ for some $\psi$, where $\pi$ denotes Frobenius. Set $\deg_i0=\infty$, then the valuation
\begin{align*}
v:\End_\Q(A) &\to \Q\cup\{\infty\}\\
a\phi    &\mapsto \ord_p(a)+\frac{1}{4}\ord_p(\deg_i\phi)
\end{align*}
for $a\in\Q$ and $\phi\in \mathcal{O}$ is well-defined. Factoring an isogeny into its separable and inseparable components gives
\begin{equation*}
\ord_p(\deg_i\phi)=\ord_p(\deg\phi)=\ord_p(\nrd \phi),
\end{equation*}
so the map $v$ is indeed a valuation on $\End_\Q(A)$.

Denote by $\Z_{(p)}$ the localisation of $\Z$ at the ideal $(p)$. Let $\alpha\in\mathcal{O}_{(p)}=\mathcal{O}\otimes \Z_{(p)}$, then $\deg\alpha\in \Z_{(p)}$, so $\alpha$ is in the valuation ring.\\
Now let $\alpha\in\End_\Q(A)$ such that $v(\alpha)>0$, then we can write $\alpha=a\phi$, where $\phi$ is an isogeny that does not factor through any multiplication-by-$k$ map.
Then $v(\alpha)=\ord_p(a)+v(\phi)\geq0$. 
Furthermore, we have that $0\leq v(\phi)$ since the degree of a morphism must be an integer.
Also, we have that $v(\phi)\leq 3/4$ since if $\phi=\pi^k\circ\psi$, where $\pi$ is Frobenius, then $k<4$.
Hence we get $0\leq v(\phi)\leq 3/4$.
Hence we have that $\ord_p(a)\geq -3/4$ and thus $a\in\Z_{(p)}$, and finally $\alpha\in \mathcal{O}_{(p)}$.

So we have shown that $\mathcal{O}_{(p)}$ is maximal, and using \cite[Lemma 9.4.6]{VoightQuatBook}, $\mathcal{O}_p$ is maximal, and using \cite[Lemma 10.4.2]{VoightQuatBook}, we are done.
\end{proof}

\subsection{Small Non-Integer Endomorphisms}
Consider the setting of \Fref{sec:reduction_cm_abvar}: We have an abelian variety $\tilde{A}$ which is the reduction of a principal abelian variety $A$ with CM by a field $K$ of degree $2g$. There is an embedding of endomorphism rings $\End(A) \hookrightarrow \End(\tilde{A})$ and hence an embedding $\mathcal{O}_K \hookrightarrow \End(\tilde{A})$. In particalular, $A$ has \emph{maximal real multiplication} (maximal RM), that is we have an embedding $\mathcal{O}_{K_0} \hookrightarrow \End(\tilde{A})$ where $\mathcal{O}_{K_0}$ is the ring of integers of the maximal real subfield $K_0$ of $K$.

For $g=2$, we have $K_0 = \Q(\sqrt{d})$ and $K = K_0(\sqrt{r})$, where $d$ is a real quadratic fundamental discriminant and $r \in \mathcal{O}_{K_0}$ is totally negative. Let
\begin{equation*}
a = \begin{cases}
\sqrt{d} &\text{ if } d \equiv 2,3 \pmod{4}, \\
\frac{1 + \sqrt{d}}{2} &\text{ if } d \equiv 1 \pmod{4},
\end{cases}
\end{equation*}
then the principal ideal $\mathfrak{a} = (a) \subset \mathcal{O}_{K_0}$ has norm $N_{K_0/\Q}(\mathfrak{a}) = -d$ or $N_{K_0/\Q}(\mathfrak{a}) = (1-d)/4$ in $K_0$, and norm $N_{K/\Q}(\mathfrak{a}) = (N_{K_0/\Q}(\mathfrak{a}))^2 = d^2$ or $N_{K/\Q}(\mathfrak{a}) = (d-1)^2/16$ in $K$, respectively. Hence, there exists a non-integer endomorphism in $\End(\tilde{A})$ of degree $d^2$ or $(d-1)^2/16$, respectively, depending on the residue of $d \pmod{4}$.

\begin{example}
We can return to our previous example of \Fref{sec:genus_2_examples}. The totally real subfield $K_0$ of the cyclic quartic CM field $K$ had discriminant $d = 5$ and so the Jacobians $\Jac(\tilde{C}_1)$ and $\Jac(\tilde{C}_2)$ both have a nontrivial non-integer endomorphism of degree 1.
\end{example}

\subsection{Superspecial Surfaces with Small Endomorphisms}
Given the structure of the endomorphism ring of reducible surfaces earlier in this section, we are able to write down reducible surfaces that are products of elliptic curves with small endomorphisms, see also \cite{LoveBoneh}.

\begin{definition}
Given an integer $M$, we say that an abelian variety $A$ over a field of characteristic $p$ is \emph{$M$-small} if there exists $\alpha\in\End(A)$ with $\deg\alpha\leq M$ such that $\alpha$ is not multiplication by integer.
\end{definition}

This allows us to identify obvious superspecial abelian surfaces with small endomorphisms; such as those that are products of $M$-small supersingular elliptic curves, and their close neighbours in the isogeny graph.

\begin{lemma}
If two supersingular elliptic curves $E_1$ and $E_2$ are $M$-small, then we have that the superspecial abelian surface $E_1\times E_2$ is $M^2$-small.
\end{lemma}
\begin{proof}
Let $\alpha_i \in \End(E_i)$ such that $\deg(\alpha_i) < M$. Then $\alpha = \alpha_1 \times \alpha_2 \in \End(E_1 \times E_2)$ has degree $\deg(\alpha)<M^2$.
\end{proof}

\begin{lemma}
Let $A$ be a superspecial abelian surface such that there exist $M$-small curves $E_1$ and $E_2$ and an isogeny $\phi: A \to E_1 \times E_2$ such that $\deg(\phi) = N$. Then $A$ is $MN^2$-small.
\end{lemma}
\begin{proof}
Let $\alpha_i\in \End(E_i)$ such that $\deg(\alpha_i) < M$. By composing the isogenies as $\hat{\phi} \circ (\alpha_1 \times \alpha_2) \circ \phi \in \End(A)$, we find an endomorphism of degree $MN^2$.
\end{proof}

\printbibliography

\end{document}